\documentclass{amsart}
\usepackage{amsmath, amsthm, amsfonts, amssymb, txfonts, graphicx, hyperref}
\theoremstyle{plain}
\newtheorem{thrm}{Theorem}[section]
\newtheorem{lmm}[thrm]{Lemma}
\newtheorem{prpstn}[thrm]{Proposition}

\newtheorem*{cnjctr}{Conjecture}
\newtheorem{hypthss}{Hypothesis}

\numberwithin{sblmm}{thrm}
\numberwithin{equation}{section}

\begin{document}
\title{Almost-prime $k$-tuples}
\author{James Maynard}
\address{Mathematical Institute, 24–-29 St Giles', Oxford, OX1 3LB}
\email{maynard@math.ox.ac.uk}
\thanks{Supported by EPSRC Doctoral Training Grant EP/P505216/1 }
\date{}
\subjclass[2010]{11N05, 11N35, 11N36}
\begin{abstract}
Let $k\ge 2$ and $\Pi(n)=\prod_{i=1}^k(a_in+b_i)$ for some integers $a_i, b_i$ ($1\le i\le k$). Suppose that $\Pi(n)$ has no fixed prime divisors. Weighted sieves have shown for infinitely many integers $n$ that $\Omega(\Pi(n))\le r_k$ holds for some integer $r_k$ which is asymptotic to $k\log{k}$. We use a new kind of weighted sieve to improve the possible values of $r_k$ when $k\ge 4$.
\end{abstract}
\maketitle
\section{Introduction}
We consider a set of integer linear functions
\begin{equation}
L_i(x)=a_ix+b_i,\qquad i\in\{1,\dots,k\}.
\end{equation}
We say such a set of functions is \textit{admissible} if their product has no fixed prime divisor. That is, for every prime $p$ there is an integer $n_p$ such that none of $L_i(n_p)$ are a multiple of $p$. We are interested in the following conjecture.
\begin{cnjctr}[Prime $k$-tuples Conjecture]
Given an admissible set of integer linear functions $L_i(x)$ ($i\in\{1,\dots,k\}$), there are infinitely many integers $n$ for which all the $L_i(n)$ are prime.
\end{cnjctr}
With the current technology it appears impossible to prove any case of the prime $k$-tuples conjecture for $k\ge 2$.

Although we cannot prove that the functions are simultaneously prime infinitely often, we are able to show that they are \textit{almost prime} infinitely often, in the sense that their product has only a few prime factors. This was most notably achieved by Chen \cite{Chen} who showed that there are infinitely many primes $p$ for which $p+2$ has at most 2 prime factors. His method naturally generalises to show that for a pair of admissible functions the product $L_1(n)L_2(n)$ has at most 3 prime factors infinitely often.

Similarly sieve methods can prove analogous results for any $k$. We can show that the product of $k$ admissible functions $\Pi(n):=L_1(n)\dots L_k(n)$ has at most $r_k$ prime factors infinitely often, for some explicitly given value of $r_k$. We see that the prime $k$-tuples conjecture is equivalent to showing we can have $r_k=k$ for all $k$. The current best values of $r_k$ grow asymptotically like $k\log{k}$ and explicitly for small $k$ we can take $r_2=3$ (Chen, \cite{Chen}), $r_3=8$ (Porter, \cite{Porter}), $r_4=12$, $r_5=16$, $r_6=20$ (Diamond and Halberstam \cite{DiamondHalberstam}), $r_7=24$, $r_8=28$, $r_9=33$, $r_{10}=38$ (Ho and Tsang, \cite{HoTsang}). Heath-Brown \cite{HeathBrown} showed that infinitely often there are $k$-tuples where all the functions $L_i$ have individually at most $C\log{k}$ prime factors, for an explicit constant $C$.

A different approach was taken by Goldston, Pintz and Y\i ld\i r\i m \cite{GPY} in their work on small gaps between primes. Under the Elliot-Halberstam conjecture, they showed that there are infinitely many $n$ for which at least two of $n$, $n+4$, $n+6$, $n+10$, $n+12$, $n+16$ are prime. Thus there must be at least one specific $2$-tuple where both functions are prime infinitely often if the Elliot-Halberstam conjecture holds.
\section{Statement of Results}
Our main result is
\begin{thrm}\label{thrm:MainTheorem}
Given a set of $k$ admissible linear functions, for infinitely many $n\in\mathbb{N}$ the product $\Pi(n)$ has at most $r_k$ prime factors, where $r_k$ is given in Table \ref{Table:MainResultsTable} below.
\begin{table}[h!b]
\begin{center}
\caption{Bounds for $\Omega(\Pi(n))$}
\label{Table:MainResultsTable}
\begin{tabular}{|c|c c c c c c c c|}
\hline
$k$ & 3 & 4 & 5 & 6 & 7 & 8 & 9 & 10\\
\hline
$r_k$ & 8 & 11 & 15 & 18 & 22 & 26 & 30 & 34\\
\hline
\end{tabular}
\end{center}
\end{table}
\end{thrm}
Theorem \ref{thrm:MainTheorem} improves the previous best known bounds for $k\ge 4$, which were obtained by Diamond and Halberstam \cite{DiamondHalberstam} for $4\le k\le 6$ and by Ho and Tsang \cite{HoTsang} for $7\le k\le 10$. We fall just short of proving $r_k\le 7$ for $k=3$, and so fail to improve upon a result of Porter \cite{Porter}. This comparison is shown in Table \ref{Table:ComparisonTable}.
\begin{table}[h]
\begin{center}
\caption{Bounds for $\Omega(\Pi(n))$}
\label{Table:ComparisonTable}
\begin{tabular}{|c|cccccccc|}
\hline
$k$ & 3 & 4 & 5 & 6 & 7 & 8 & 9 & 10\\
\hline
Previous best bound & 8 & 12 & 16 & 20 & 24 & 28 & 33 & 38\\
New bound & 8 & 11 & 15 & 18 & 22 & 26 & 30 & 34\\
\hline
\end{tabular}
\end{center}
\end{table}
We prove these results using a sieve which is a combination of a weighted sieve similar to Selberg's $\Lambda^2\Lambda^-$ sieve (see \cite{Selberg}), and the Graham-Goldston-Pintz-Y\i ld\i r\i m  sieve (see \cite{GGPY}) used to count numbers with a specific number of prime factors. 

We note that for $k$ large our method only improves lower order terms, and so we do not improve the asymptotic bound $r_k\sim k\log{k}$.

In a forthcoming paper \cite{Maynard}, we will also improve the bound when $k=3$, using an argument based on the Diamon-Halberstam-Richert sieve rather than Selberg's sieve.
\section{Key Ideas}
We wish to show that for any sufficiently large $N$ we have
\begin{equation}\label{eq:MainSum}
\sum_{N\le n\le 2N}(c-\Omega(\Pi(n)))\left(\sum_{d|\Pi(n)}\lambda_d\right)^2>0
\end{equation}
for some real numbers $\lambda_d$ and some constant integer $c>0$. From this it is clear that there must be some $n\in[N,2N]$ such that $\Omega(\Pi(n))\le c$. Since this is true for all sufficiently large $N$, it follows that there are infinitely many integers $n$ such that $\Omega(\Pi(n))\le c$.

The work of Heath-Brown \cite{HeathBrown} and Ho and Tsang \cite{HoTsang} considered a similar sum, but used the divisor function $d(\Pi(n))$ instead of the number-of-prime-factors function $\Omega$. Using the divisor function has the advantage that there are stronger level-of-distribution results available, but we find that this is outweighed by the fact that the $\Omega$ function is relatively much smaller than the divisor function on numbers with many prime factors.

The $\Omega$ function has Bombieri-Vinogradov style equidistribution results (as shown by Motohashi \cite{Motohashi}), and so we would expect we should be able to estimate the above sum directly, in a method similar to Heath-Brown \cite{HeathBrown} or Selberg \cite{Selberg} when they considered the divisor function instead. We encounter some technical difficulties when attempting to translate this argument, however.

Instead we express $\Omega(n)$ as a weighted sum over small prime factors (as in the weighted sieve method of Diamond and Halberstam \cite{DiamondHalberstam}) and a remaining positive contribution which we split up depending on the number of prime factors of each of the $L_j(n)$. 

Diamond and Halberstam used a weighted sieve. The method relied on the fact for $n$ square-free we have the inequality
\begin{equation}
\Omega(n)\le \sum_{\substack{p|n\\ p\le y}}\left(1-\frac{\log{p}}{\log{y}}\right)+\frac{\log{n}}{\log{y}}.
\end{equation}
We note that this inequality is strict if $n$ has a prime factor which is larger than $y$. This results in a loss in the argument which has a noticeable effect when we apply this to $k$-tuples when $k$ is small. Assuming that $y\ge n^{1/2}$ and $n$ square-free we can write instead an equality
\begin{equation}
\Omega(n)=\sum_{\substack{p|n\\ p\le y}}\left(1-\frac{\log{p}}{\log{y}}\right)+\frac{\log{n}}{\log{y}}+\sum_{r=1}^\infty \chi_r(n),
\end{equation}
where
\begin{align}
\chi_r(n)&=\begin{cases}
1-\frac{\log{p_r}}{\log{y}},\qquad&n=p_1\dots p_r\text{ with }p_1\le p_2\le \dots\le p_{r-1}\text{ and }y< p_r,\\
0,&\text{otherwise,}
\end{cases}\nonumber\\
&=\begin{cases}
-\left(\frac{\log{n}}{\log{y}}-1-\sum_{i=1}^{r-1}\frac{\log{p_i}}{\log{y}}\right),\qquad &n=p_1\dots p_r\text{ with }p_1\le p_2\le \dots\le p_r\\
&\qquad\text{and }y< p_r,\\
0,&\text{otherwise.}
\end{cases}
\end{align}
For fixed $r$ we can evaluate Selberg-type weighted sums over $\chi_r(L_i(n))$ using the method of Graham, Goldston, Pintz and Y\i ld\i r\i m in \cite{GGPY} as an extension of the original GPY method. We note that the contribution from $\chi_r(n)$ is always negative, so we can obtain a lower bound by simply omitting terms when $r>h$ for some constant $h$. The contribution of the $\chi_r$ terms decreases quickly with $r$, and so we in practice only need to calculate the contribution when $r$ is small (in this paper we only consider the contributions of $\chi_r$ when $r\le 4$). This is the key difference in our approach to previous methods, and allows us to obtain the improvements given by Theorem \ref{thrm:MainTheorem}.
\newpage
\section{Initial Considerations}
We adopt similar notation to that of Graham, Goldston, Pintz and Y\i ld\i r\i m in \cite{GGPY}.

Let $\mathcal{L}=\{L_1,L_2,\dots,L_k\}$ be an admissible $k$-tuple of linear functions. We define
\begin{align}
\Pi(n)&=\prod_{i=1}^k L_i(n)=(a_1n+b_1)\dots(a_k n+b_k),\\
\nu_p(\mathcal{L})&=\#\{1\le n\le p:\Pi(n)\equiv 0\pmod{p}\}.
\end{align}
We note that admissibility is equivalent to the condition
\begin{equation}\nu_p(\mathcal{L})<p\qquad \text{for all primes $p$.}\end{equation}
We also see that $v_p(\mathcal{L})\le k$ for all primes $p$, and so the above condition holds automatically for $p>k$.

For technical reasons we adopt a normalisation of our linear functions, as done originally by Heath-Brown in \cite{HeathBrown}. Since we are only interested in the showing any admissible $k$-tuple has at most $r_k$ prime factors infinitely often (for some explicit $r_k$), by considering the functions $L_i(An+B)$ for suitably chosen constants $A$ and $B$, we may assume without loss of generality that our functions satisfy the following hypothesis.
\begin{hypthss}\label{hypthss:Normalised}
$\mathcal{L}=\{L_1,\dots,L_k\}$ is an admissible $k$-tuple of linear functions. The functions $L_i(n)=a_i n+b_i$ ($1\le i \le k$) are distinct with $a_i>0$. Each of the coefficients $a_i$ is composed of the same primes, none of which divides the $b_j$. If $i\ne j$, then any prime factor of $a_i b_j-a_j b_i$ divides each of the $a_l$.
\end{hypthss}
For a set of linear functions satisfying Hypothesis \ref{hypthss:Normalised} we define
\begin{equation}
A=\prod_{i=1}^k a_i.
\end{equation}
We note that in this case
\begin{equation}
\nu_p(\mathcal{L})=\begin{cases}
0,\qquad &p|A,\\
k, &p\nmid A.
\end{cases}
\end{equation}
We also define the \textit{singular series} $\mathfrak{S}(\mathcal{L})$ of $\mathcal{L}$ when $\mathcal{L}$ satisfies Hypothesis \ref{hypthss:Normalised}.
\begin{equation}
\mathfrak{S}(\mathcal{L})=\prod_{p|A}\left(1-\frac{1}{p}\right)^{-k}\prod_{p\nmid A}\left(1-\frac{k}{p}\right)\left(1-\frac{1}{p}\right)^{-k}.
\end{equation}
We note that $\mathfrak{S}(\mathcal{L})$ is positive.

As is common with the Selberg sieve, for some parameter $R_2$ we impose the condition
\begin{equation}\label{eq:LambdaConds}
\lambda_d=0\qquad\text{if $d\ge R_2$ or $d$ not square-free or $(d,A)\ne 1$.}
\end{equation}
We wish to choose the $\lambda_d$ to maximize the sum \eqref{eq:MainSum}, but this will be difficult to do optimally. We proceed by reparameterising the form in $\lambda_d$ into new variables $y_r$ and $y_r^*$ which will almost diagonalise it. We define
\begin{align}
y_r=\mu(r)f_1(r)\sideset{}{'}{\sum}_{d}\frac{\lambda_{dr}}{f(dr)},\\
y_r^*=\mu(r)f^*_1(r)\sideset{}{'}{\sum}_{d}\frac{\lambda_{dr}}{f^*(dr)},
\end{align}
where here and from now on, the $'$ by the summation indicates that the sum is over all values of the indices which are square-free and coprime to $A$. For square-free $d$ coprime to $A$, the functions $f$, $f_1$, $f^*$ and $f_1^*$ are defined by
\begin{align}
f(d)&=\prod_{p|d}\frac{p}{k},\\
f_1(d)&=(f*\mu)(d)=\prod_{p|d}\frac{p-k}{k},\\
f^*(d)&=\prod_{p|d}\frac{p-1}{k-1},\\
f_1^*(d)&=(f^**\mu)(d)=\prod_{p|d}\frac{p-k}{k-1}.
\end{align}
We note that by M\"obius inversion we have
\begin{equation}\label{eq:LambdaDef}
\lambda_d=\mu(d)f(d)\sideset{}{'}{\sum}_r\frac{y_{rd}}{f_1(rd)}.
\end{equation}
Thus the $\lambda_d$ (and hence also the $y_r^*$) are defined uniquely by a choice of the $y_r$. The conditions \eqref{eq:LambdaConds} will be satisfied if the same conditions apply to the $y_r$.

For some polynomial $P$ (to be determined later), we choose
\begin{equation}\label{eq:YDef}
y_r=\begin{cases}
\mu^2(r)\mathfrak{S}(\mathcal{L})P\left(\frac{\log{R_2/r}}{\log{R_2}}\right),\qquad &\text{if $r\le R_2$ and $(r,A)=1$,}\\
0,&\text{otherwise.}
\end{cases}
\end{equation}
We now turn our attention to the proof of the theorem.
\section{Proof of Theorem}
We consider the sum
\begin{equation}S=S(\nu;N,R_1,R_2,\mathcal{L})=\sum_{N\le n \le 2N}w(n)\Lambda^2(n),\end{equation}
where
\begin{align}
w(n)&=\nu-\sum_{p|\Pi(n)}\left(1-\frac{\log{p}}{\log{R_1}}\right),\\
\Lambda^2(n)&=\Big(\sum_{\substack{d|\Pi(n)\\ d\le R_2}}\lambda_d\Big)^2.
\end{align}
We note that if $\Pi(n)$ is square-free then
\begin{equation}
w(n)=\nu-\Omega(\Pi(n))+\frac{\log{\Pi(n)}}{\log{R_1}}.
\end{equation}
We see that for $n\in [N,2N]$ and some fixed $h\in\mathbb{Z}_{>0}$ we have
\begin{align}
w(n)&=\nu-\sum_{j=1}^k\sum_{p|L_j(n)}\left(1-\frac{\log{p}}{\log{R_1}}\right)\nonumber\\
&\ge \nu-\sum_{j=1}^k\sum_{\substack{p|L_j(n)\\p\le R_1\text{ or }\Omega(L_j(n))\le h}}\left(1-\frac{\log{p}}{\log{R_1}}\right)\nonumber\\
&\ge\nu-\sum_{j=1}^k\sum_{\substack{p|L_j(n)\\p\le R_1}}\left(1-\frac{\log{p}}{\log{R_1}}\right)+\sum_{j=1}^k\sum_{r=1}^{h}\chi_r(L_j(n)),
\end{align}
where
\begin{equation}
\chi_r(n)=\begin{cases}
\frac{\log{N}}{\log{R_1}}-1-\sum_{i=1}^{r-1}\frac{\log{p_i}}{\log{R_1}}, &\text{if }n=p_1\dots p_r\text{ with }\\
&\quad n^\epsilon<p_1<\dots<p_{r-1}\le n^{\log{R_1}/\log{N}}<p_r\\
0,&\text{otherwise.}
\end{cases}
\end{equation}
Thus
\begin{align}
\sum_{\substack{N\le n\le 2N\\ \Pi(n)\text{ square-free}}}\left(\nu-\Omega(\Pi(n))+\frac{\log{\Pi(n)}}{\log{R_1}}\right)\Lambda^2(n)&= S-S'\nonumber\\
&\ge vS_0-S'-T_0+\sum_{j=1}^k \sum_{r=1}^h T_{r,j},
\end{align}
where
\begin{align}
S_0&=\sum_{N\le n\le 2N}\Lambda^2(n),\\
S'&=\sum_{\substack{N\le n\le 2N\\\Pi(n)\text{ not square-free}}}w(n)\Lambda^2(n),\\
T_0&=\sum_{N\le n\le 2N}\sum_{\substack{p|\Pi(n)\\p\le R_1}}\left(1-\frac{\log{p}}{\log{R_1}}\right)\Lambda^2(n),\\
T_{r,j}&=\sum_{N\le n\le 2N}\chi_r(L_j(n))\Lambda^2(n).
\end{align}
We can evaluate $S_0$, $S'$, $T_0$ and $T_i$ using weighted forms of the Selberg sieve. We state the results here and prove them in the following sections. To ease notation we now fix as constants
\begin{equation}
r_1=\frac{\log{R_1}}{\log{N}},\qquad r_2=\frac{\log{R_2}}{\log{N}}.
\end{equation}
We view $r_1$, $r_2$, $k$, $A$ and our polynomial $P$ as fixed, and so any constants implied by the use of $O$ or $\ll$ notation may depend on these quantities without explicit reference.
\newpage
\begin{prpstn}\label{prpstn:MainTermResult}\label{prpstn:T0}
Let $\mathcal{L}$ satisfy Hypothesis \ref{hypthss:Normalised}. Let $W_0:[0,r_1/r_2]:\rightarrow \mathbb{R}_{\ge 0}$ be a piecewise smooth non-negative function. Let $\lambda_d, y_d$ be as given in \eqref{eq:LambdaDef} and \eqref{eq:YDef}. Assume that $r_1\ge r_2$. Then there exists a constant $C$ such that if $R_1R_2^2\le N(\log{N})^{-C}$ then we have
\begin{align*}
\sum_{N\le n\le 2N}\left(\sum_{\substack{p|\Pi(n)\\p\le R_1}}W_0\left(\frac{\log{p}}{\log{R_2}}\right)\right)\left(\sum_{\substack{d|\Pi(n)\\d\le R_2}}\lambda_d\right)^2&=\frac{\mathfrak{S}(\mathcal{L})N(\log{R_2})^{k}}{(k-1)!}J_0\\
&\qquad+O_{W_0}\left(N(\log{N})^{k-1}(\log{\log{N}})^2\right),
\end{align*}
where
\begin{align*}
J_0&=J_{01}+J_{02}+J_{03},\\
J_{01}&=k\int_0^{1}\frac{W_0(y)}{y}\int_0^{1-y}(P(1-x)-P(1-x-y))^2 x^{k-1}dx dy,\\
J_{02}&=k\int_0^{1}\frac{W_0(y)}{y}\int_{1-y}^1 P(1-x)^2 x^{k-1}dx dy,\\
J_{03}&=k\int_{1}^{r_1/r_2}\frac{W_0(y)}{y}\int_0^{1}P(1-x)^2x^{k-1}dx dy.
\end{align*}
\end{prpstn}
\begin{prpstn}\label{prpstn:ErResult}\label{prpstn:Ti}
Given $\epsilon>0$ and $r\in\mathbb{Z}_{>0}$, let
\begin{equation*}\mathcal{A}_r:=\left\{x\in[0,1]^{r-1}:\epsilon<x_1<\dots<x_{r-1},\sum_{i=1}^{r-1}x_i<\min(1-r_2,1-x_{r-1})\right\}.\end{equation*}
Let $W_r:[0,1]^{r-1}\rightarrow\mathbb{R}_{\ge 0}$ be a piecewise smooth function supported on $\mathcal{A}_r$ such that
\begin{equation*}\frac{\partial}{\partial x_j}W_r(x)\ll W_r(x)\qquad \text{uniformly for $x\in \mathcal{A}_r$.}\end{equation*}
Let
\begin{equation*}\beta_r(n)=\begin{cases}
W_r\left(\frac{\log{p_1}}{\log{n}},\dots,\frac{\log{p_{r-1}}}{\log{n}}\right),\qquad &\text{$n=p_1p_2\dots p_r$ with $p_1<\dots<p_r$,}\\
0,&\text{otherwise,}
\end{cases}\end{equation*}
Then there is a constant $C$ such that if $R_2^2\le N^{1/2}(\log{N})^{-C}$, we have
\begin{equation*}\sum_{N\le n\le 2N}\beta_r(L_j(n))\left(\sum_{\substack{d|\Pi(n)\\ d\le R_2}}\lambda_d\right)^2=\frac{\mathfrak{S}(\mathcal{L})N(\log{R_2})^{k+1}}{(k-2)!(\log{N})}J_r+O_{W_r}\left(N(\log\log{N})^{r}(\log{N})^{k-1}\right),\end{equation*}
where
\begin{align*}
J_r&=\int_{(x_1,\dots,x_{r-1})\in \mathcal{A}_r}\frac{ W_r(x_1,\dots,x_{r-1})I_1(r_2^{-1}x_1,\dots,r_2^{-1}x_{r-1})}{\left(\prod_{i=1}^{r-1}x_i\right)\left(1-\sum_{i=1}^{r-1}x_i\right)}d x_1\dots dx_{r-1},\\
I_1&=\int_0^1\left(\sum_{J\subset\{1,\dots,r-1\}}(-1)^{|J|}\tilde{P}^+(1-t-\sum_{i\in J}x_i) \right)^2t^{k-2}d t,\\
\tilde{P}^+(x)&=\begin{cases}
\int_0^x P(t)d t,\qquad &x\ge 0\\
0,&\text{otherwise.}
\end{cases}
\end{align*}
\end{prpstn}
\begin{prpstn}\label{prpstn:SquareFree}\label{prpstn:S'}
There exists a constant $C$ such that if $R_2^2\le N^{1/2}(\log{N})^{-C}$ then
\begin{equation*}\sum_{\substack{N\le n\le 2N\\\Pi(n)\textnormal{ not square-free}}}\left(\sum_{\substack{d|\Pi(n)\\d\le R_2}}\lambda_d\right)^2\ll N(\log{N})^{k-1}\log\log{N}.\end{equation*}
\end{prpstn}
We also quote a result \cite{GGPY}[Theorem 7] which is based on the original result of Goldston, Pintz and Y\i ld\i r\i m in \cite{GPY}.
\begin{prpstn}\label{prpstn:S0}
There is a constant $C$ such that if $R_2^2\le N(\log{N})^{-C}$, we have
\[\sum_{N\le n\le 2N}\left(\sum_{\substack{d|\Pi(n)\\d\le R_2}}\lambda_d\right)^2=\frac{\mathfrak{S}(\mathcal{L})N(\log{R_2})^k}{(k-1)!}J+O\left(N(\log{N})^{k-1}\right)\]
where
\[J=\int_0^1P(1-t)^2t^{k-1}d t .\]
\end{prpstn}

Using Propositions \ref{prpstn:T0}, \ref{prpstn:Ti}, \ref{prpstn:S0} and \ref{prpstn:S'} we can now bound our sum $S$ in terms of the integers $k$ and $h$ and the polynomial $P$. For some $\epsilon>0$ we choose
\begin{equation}r_1=\frac{1}{2}+\epsilon,\qquad r_2=\frac{1}{4}-\epsilon,\end{equation}
so that the conditions of all the propositions are satisfied.

Proposition \ref{prpstn:S0} gives the size of $S_0$ immediately.

Using Proposition \ref{prpstn:S'} we have
\begin{align}
S'&=\sum_{\substack{N\le n \le 2N\\ \Pi(n)\text{ not square-free}}}w(n)\Lambda^2(n)\nonumber\\
&\le \sum_{\substack{N\le n \le 2N\\ \Pi(n)\text{ not square-free}}}\left(\nu+\frac{\log{\Pi(n)}}{\log{R_1}}\right)\Lambda^2(n)\nonumber\\
&\le \sum_{\substack{N\le n \le 2N\\ \Pi(n)\text{ not square-free}}}\left(\nu+\frac{k+\epsilon}{r_1}\right)\Lambda^2(n)\nonumber\\
&\ll N(\log{N})^{k-1}\log\log{N}.
\end{align}

To estimate $T_0$ and the $T_{r,j}$ we choose
\begin{align}
W_0(x)&=1-\frac{r_2}{r_1}x,\\
W_j(x_1,\dots,x_{j-1})&=\begin{cases}
\frac{1}{r_1}-1-\frac{1}{r_1}\sum_{i=1}^{j-1}x_i, &\epsilon<x_1<\dots<x_{j-1}\\
&\qquad\text{ and }\sum_{i=1}^{r-1}x_i<1-r_1\\
0,&\text{otherwise,}\end{cases}
\end{align}
which satisfy the conditions of Propositions \ref{prpstn:T0} and \ref{prpstn:Ti} respectively.

By Proposition \ref{prpstn:T0} we have
\begin{align}
T_0&=\sum_{N\le n\le 2N}\left(\sum_{\substack{p|\Pi(n)\\p\le R_1}}W_0\left(\frac{\log{p}}{\log{R_2}}\right)\right)\left(\sum_{\substack{d|\Pi(n)\\d\le R_2}}\lambda_d\right)^2\nonumber\\
&=\frac{\mathfrak{S}(\mathcal{L})N(\log{R_2})^{k}}{(k-1)!}J_0+O\left(N(\log{N})^{k-1}\log{\log{N}}\right)
\end{align}
where
\begin{align}
J_0&=J_{01}+J_{02}+J_{03},\\
J_{01}&=k\int_0^{1}\frac{r_1-r_2y}{r_1y}\int_0^{1-y}(P(1-x)-P(1-x-y))^2 x^{k-1}dx dy,\\
J_{02}&=k\int_0^{1}\frac{r_1-r_2y}{r_1y}\int_{1-y}^1 P(1-x)^2 x^{k-1}dx dy,\\
J_{03}&=k\int_{1}^{r_1/r_2}\frac{r_1-r_2y}{r_1y}\int_0^{1}P(1-x)^2x^{k-1}dx dy.
\end{align}
By Proposition \ref{prpstn:Ti} we have
\begin{align}
T_{r,j}&=\sum_{N\le n\le 2N}\chi_r(L_j(n))\Lambda^2(n)\nonumber\\
&=\sum_{N\le n\le 2N}\beta_r(L_j(n))\Lambda^2(n)\nonumber\\
&=\frac{\mathfrak{S}(\mathcal{L})N(\log{R_2})^{k+1}}{(k-2)!(\log{N})}J_r+O_r\left(N(\log\log{N})^{r+1}(\log{N})^{k-1}\right),
\end{align}
where
\begin{align}
\beta_r(n)&=\begin{cases}
W_r\left(\frac{\log{p_1}}{\log{n}},\dots,\frac{\log{p_{r-1}}}{\log{n}}\right),\qquad &\text{$n=p_1p_2\dots p_r$ with $p_1<\dots<p_r$,}\\
0,&\text{otherwise,}\end{cases}\\
J_r&=\int_{(x_1,\dots,x_{r-1})\in \mathcal{A}_r}\frac{ W_r(x_1,\dots,x_{r-1})I_1(r_2^{-1}x_1,\dots,r_2^{-1}x_{r-1})}{\left(\prod_{i=1}^{r-1}x_i\right)\left(1-\sum_{i=1}^{r-1}x_i\right)}d x_1\dots dx_{r-1}.
\end{align}
Therefore we see that
\begin{align}
\nu S_0-S'+T_0+\sum_{j=1}^k\sum_{r=1}^h T_{r,j}&=\frac{N\mathfrak{S}(\mathcal{L})(\log{R_2})^k}{(k-1)!}\left(\nu J-J_0+r_2k(k-1)(\sum_{r=1}^{h}J_r)\right)\\
&\qquad+O\left(\frac{N(\log{N})^{k}}{\log\log{N}}\right).\nonumber
\end{align}
Therefore we put
\begin{equation}\nu=\frac{J_0-r_2k(k-1)(\sum_{r=1}^h J_r)}{J}+\epsilon.\end{equation}
We then see that for any $N$ sufficiently large we have
\begin{equation}\nu S_0-S'-T_0+\sum_{j=1}^k\sum_{r=1}^hT_{r,j}>0.\end{equation}
Thus we have
\begin{equation}\Omega(\Pi(n))\le \left\lfloor\frac{J_0-r_2k(k-1)(\sum_{r=1}^{h}J_r)}{J}+\frac{k}{r_1}+2\epsilon\right\rfloor\end{equation}
infinitely often.

With these fixed, given $k,h$ and a polynomial $P$ we obtain a bound on $\Omega(\Pi(n))$. To make calculations feasible we choose $h=3$ (except we take $h=4$ when $k=10$). Numerical experiments indicate that the bounds of Theorem 1 cannot be improved by increasing $h$ except possibly when $k=5$. 

We can now explicitly write down the integrals $J_1$, $J_2$ and $J_3$, splitting the integral up depending on whether $\tilde{P}^+$ is positive or not. We put
\begin{equation}
\tilde{P}(x)=\int_0^xP(t)dt.
\end{equation}
Then we have that
 \begin{equation}
J_1=\left(\frac{1-r_1}{r_1}\right)\int_0^1\tilde{P}(1-x)^2x^{k-2}dx+O(\epsilon).
\end{equation}
Similarly
\begin{equation}J_2=J_{21}+J_{22}+J_{23}+O(\epsilon),\end{equation}
where
\begin{align}
J_{21}&=\int_{0}^1\frac{1-r_1-r_2y}{r_1y(1-r_2y)}\int_0^{1-y}\left(\tilde{P}(1-x)-\tilde{P}(1-x-y)\right)^2x^{k-2}d x d y,\\
J_{22}&=\int_{0}^1\frac{1-r_1-r_2y}{r_1y(1-r_2y)}\int_{1-y}^1\tilde{P}(1-x)^2x^{k-2}d x d y,\\
J_{23}&=\int_1^{(1-r_1)/r_2}\frac{1-r_1-r_2y}{r_1y(1-r_2y)}\int_{0}^1\tilde{P}(1-x)^2x^{k-2}d x d y.
\end{align}
Finally
\begin{equation}
J_3=J_{31}+J_{32}+J_{33}+J_{34}+J_{35}+J_{36}+J_{37}+J_{38}+O(\epsilon),\\
\end{equation}
where
\begin{align}
J_{31}&=\int_1^{(1-r_1)/2r_2}\int_y^{(1-r_1)/r_2-y}\frac{1-r_1-r_2(y+z)}{r_1y z(1-r_2(y+z))}\int_{0}^1\tilde{P}(1-x)^2x^{k-2}d x d z d y,\\
J_{32}&=\int_{0}^1\int_y^{(1-r_1)/r_2-y}\frac{1-r_1-r_2(y+z)}{r_1y z(1-r_2(y+z))}\int_{1-y}^1\tilde{P}(1-x)^2x^{k-2}d x d z d y,\\
J_{33}&=\int_{0}^{1}\int_{1}^{(1-r_1)/r_2-y}\frac{1-r_1-r_2(y+z)}{r_1y z(1-r_2(y+z))}\nonumber\\
&\qquad\qquad\int_{0}^{1-y}\left(\tilde{P}(1-x)-\tilde{P}(1-x-y)\right)^2x^{k-2}d x dz d y,\\
J_{34}&=\int_{0}^1\int_y^1\frac{1-r_1-r_2(y+z)}{r_1y z(1-r_2(y+z))}\nonumber\\
&\qquad\qquad\int_{1-z}^{1-y}\left(\tilde{P}(1-x)-\tilde{P}(1-x-y)\right)^2x^{k-2}d x dz d y,
\end{align}
\begin{align}
J_{35}&=\int_{1/2}^{1}\int_{y}^1\frac{1-r_1-r_2(y+z)}{r_1y z(1-r_2(y+z))}\nonumber\\
&\qquad\qquad\int_{0}^{1-z}\left(\tilde{P}(1-x)-\tilde{P}(1-x-y)-\tilde{P}(1-x-z)\right)^2x^{k-2}d x d z d y,\\
J_{36}&=\int_{0}^{1/2}\int_{1-y}^1\frac{1-r_1-r_2(y+z)}{r_1y z(1-r_2(y+z))}\nonumber\\
&\qquad\qquad\int_{0}^{1-z}\left(\tilde{P}(1-x)-\tilde{P}(1-x-y)-\tilde{P}(1-x-z)\right)^2x^{k-2}d x d z d y,\\
J_{37}&=\int_{0}^{1/2}\int_{y}^{1-y}\frac{1-r_1-r_2(y+z)}{r_1y z(1-r_2(y+z))}\nonumber\\
&\qquad\qquad\int_{1-y-z}^{1-z}\left(\tilde{P}(1-x)-\tilde{P}(1-x-y)-\tilde{P}(1-x-z)\right)^2x^{k-2}d x d z d y,\\
J_{38}&=\int_{0}^{1/2}\int_{y}^{1-y}\frac{1-r_1-r_2(y+z)}{r_1y z(1-r_2(y+z))}\int_0^{1-y-z}\biggl(\tilde{P}(1-x)-\tilde{P}(1-x-y)\nonumber\\
&\qquad\qquad-\tilde{P}(1-x-z)+\tilde{P}(1-x-y-z)\biggr)^2x^{k-2}d x d z d y.
\end{align}
We now have explicit representations of $J$, $J_0$, $J_1$, $J_2$ and $J_3$. We can calculate these by numerical integration given $k$ and a polynomial $P$.

Table \ref{Table:ValuesTable} gives close to optimal polynomials for $3\le k\le 10$ and the corresponding bounds obtained if we take $\epsilon$ sufficiently small. These give the results claimed in Theorem \ref{thrm:MainTheorem} except for $k=10$.
\begin{table}[hb]
\begin{center}
\caption{Bounds for $\Omega(\Pi(n))$}
\label{Table:ValuesTable}
\begin{tabular}{|c|c|c|}
\hline
$k$ & Bound on $\Omega(\Pi(n))$ & Polynomial $P(x)$\\
\hline
3 & 8.220\dots &  $1+14x$\\
4 & 11.653\dots & $1+22x$\\
5 & 15.306\dots & $1+33x$\\
6 & 18.936\dots & $1+10x+40x^2$\\
7 & 22.834\dots & $1+10x+60x^2$\\
8 & 26.860\dots & $1+10x+80x^2$\\
9 & 30.942\dots & $1+30x+300x^3$\\
10 & 35.158\dots & $1+35x-10x^2+400x^3$ \\
\hline
\end{tabular}
\end{center}
\end{table}

For $k=10$ we find an improvement if we also include the contribution when one of the $L_i(n)$ has 4 prime factors (we omit the explicit integrals here). In this case we choose the polynomial
\begin{equation}
P(x)=1+10x+150x^2.
\end{equation}
This gives us the bound 34.77... and so 10-tuples infinitely often have at most 34 prime factors, verifying Theorem 1.
\clearpage
\section{The quantities \texorpdfstring{$T_\delta$ and $T_\delta^*$}{T and T*}}
Before proving the propositions, we first establish some results about the quantities
\begin{align}
T_\delta&=\sideset{}{'}{\sum}_{d,e}\frac{\lambda_d \lambda_e}{f([d,e,\delta]/\delta)}, \label{eq:Tdef}\\
T_\delta^*&=\sideset{}{'}{\sum}_{d,e}\frac{\lambda_d \lambda_e}{f^*([d,e,\delta]/\delta)}.\label{eq:Tsdef}
\end{align}
Most of these results already exist in some form in the literature. These results will underlie the proof of the propositions. We note that in \cite{GGPY} Graham, Goldston, Pintz and Y\i ld\i r\i m used slightly different notation (our quantity $T^*_\delta$ is labelled $T_\delta$).

We first put $T_\delta$ and $T_\delta^*$ into an almost-diagonalised form.
\begin{lmm}\label{lmm:Tres}
We have
\begin{align*}
T_\delta&=\sideset{}{'}{\sum}_{\substack{a\\ (a,\delta)=1}}\frac{\mu^2(a)}{f_1(a)}\left(\sum_{s|\delta}\mu(s)y_{as}\right)^2,\\
T_\delta^*&=\sideset{}{'}{\sum}_{\substack{a\\ (a,\delta)=1}}\frac{\mu^2(a)}{f_1^*(a)}\left(\sum_{s|\delta}\mu(s)y^*_{as}\right)^2,
\end{align*}
where
\begin{equation*}y_a^*=\frac{\mu^2(a)a}{\phi(a)}\sideset{}{'}{\sum}_m\frac{y_{ma}}{\phi(m)}.\end{equation*}
\end{lmm}
\begin{proof}
The result for $T_\delta$ is shown, for example, in \cite{Selberg}[Page 85]. The result for $T_\delta^*$ is proven in \cite{GGPY}[Lemma 6].
\end{proof}
We now again quote a Lemma from \cite{GGPY}, which expresses the $y_a^*$ in terms of the polynomial $P$ which we used to define the variables $y_a$.
\begin{lmm}\label{lmm:y*}
Let
\begin{equation*}y_a=\begin{cases}
\mu^2(a)\mathfrak{S}(\mathcal{L})P\left(\frac{\log{R_2/a}}{\log{R_2}}\right),\qquad &\text{if $0\le a< R_2$ and $(a,A)=1$} \\
0,&\text{otherwise}\end{cases}.\end{equation*}
Then we have for $(a,A)=1$ and $a<R_2$ that
\begin{align*}
y_a^*&=\mu^2(a)\frac{\phi(A)}{A}\mathfrak{S}(\mathcal{L})(\log{R_2})\tilde{P}\left(\frac{\log{R_2/a}}{\log{R_2}}\right)+O(\log\log{R_2}),
\end{align*}
where
\begin{equation*}
\tilde{P}(x)=\int_0^xP(t)dt.
\end{equation*}
If $(a,A)\ne 1$ or $a\ge R_2$ then we have
\begin{equation*}y_a^*=0.\end{equation*}
\end{lmm}
\begin{proof}
This is proven in \cite{GGPY}[Lemma 7].
\end{proof}
We will repeatedly use the following result.
\begin{lmm}\label{lmm:SummationLemma}
For $u\ge 1$ we have
\begin{align*}\sideset{}{'}{\sum}_{a<u}\frac{\mu^2(a)}{f_1(a)}&=\frac{A}{\phi(A)}\frac{(\log{u})^{k}}{\mathfrak{S}(\mathcal{L})k!}+O((\log{2u})^{k-1}),\\
\sideset{}{'}{\sum}_{a<u}\frac{\mu^2(a)}{f_1^*(a)}&=\frac{A}{\phi(A)}\frac{(\log{u})^{k-1}}{\mathfrak{S}(\mathcal{L})(k-1)!}+O((\log{2u})^{k-2}).\end{align*}
\end{lmm}
\begin{proof}
This is follows, for example, from \cite{GGPY}[Lemma 3].
\end{proof}
In order to estimate the terms $T_\delta^*$ we wish to remove the condition $(a,\delta)=1$ in the summation over $a$, and remove the constraint caused by $y_a$ and $y^*_a$ only being supported on square-free $a$. We let
\begin{align}
P_a&=\begin{cases}
\mathfrak{S}(\mathcal{L})P\left(\frac{\log{R_2/a}}{\log{R_2}}\right),\qquad &\text{if $0\le a < R_2$}\\
0,&\text{otherwise,}
\end{cases}\label{eq:PDef}\\
P^*_a&=\begin{cases}
\frac{\phi(A)}{A}\mathfrak{S}(\mathcal{L})(\log{R_2})\tilde{P}\left(\frac{\log{R_2/a}}{\log{R_2}}\right),\qquad &\text{if $0\le a < R_2$}\\
0,&\text{otherwise,}
\end{cases}\label{eq:P*Def}
\end{align}
so that these are equal to $y_a$ and $y^*_a+O(\log\log{R_2})$ respectively when $a$ is square-free and coprime to $A$.
\begin{lmm}\label{lmm:Tdiv}
Let $(\delta,A)=1$. Then we have
\begin{align*}
T_\delta&=\sideset{}{'}{\sum}_a\frac{\mu^2(a)}{f_1(a)}\left(\sum_{s|\delta}\mu(s)P_{as}^*\right)^2+O\left(d(\delta)^2(\log{R_2})^{k-1}\log\log{R_2}\right),\\
T_\delta^*&=\sideset{}{'}{\sum}_a\frac{\mu^2(a)}{f_1^*(a)}\left(\sum_{s|\delta}\mu(s)P_{as}^*\right)^2+O\left(d(\delta)^2(\log{R_2})^{k}\log\log{R_2}\right).
\end{align*}
\end{lmm}
\begin{proof}
We only prove the result for the $T^*_\delta$ here, the result for the $T_\delta$ follows from a completely analogous argument.
We see that since $P^*_a\ll \log{R_2}$ we have
\begin{align}
T_\delta^*&=\sideset{}{'}{\sum}_{\substack{a\\(a,\delta)=1}}\frac{\mu^2(a)}{f_1^*(a)}\left(\sum_{s|\delta}\mu(s)P^*_{as}+O(\log\log{R_2})\right)^2\nonumber\\
&=\sideset{}{'}{\sum}_{\substack{a\\(a,\delta)=1}}\frac{\mu^2(a)}{f_1^*(a)}\left(\sum_{s|\delta}\mu(s)P^*_{as}\right)^2+O\left(d(\delta)^2(\log{R_2})(\log\log{R_2})\sum_{a<R_2}\frac{\mu^2(a)}{f_1^*(a)}\right).
\end{align}
By Lemma \ref{lmm:SummationLemma} the error term above is $O(d(\delta)^2(\log{R_2})^k\log\log{R_2})$.

We see that to prove the result it is sufficient to prove
\begin{equation}\sideset{}{'}{\sum}_{\substack{a\\(a,\delta)\ne 1}}\frac{\mu^2(a)}{f_1^*(a)}\left(\sum_{s|\delta}\mu(s)P_{as}^*\right)^2\ll (\log{R_2})^{k}d(\delta)^2 (\log\log{R_2}).\label{eq:NotCoprimeSmall}\end{equation}
Since all terms in the sum are non-negative, we have
\begin{equation}
\sideset{}{'}{\sum}_{\substack{a\\(a,\delta)\ne 1}}\frac{\mu^2(a)}{f_1^*(a)}\left(\sum_{s|\delta}\mu(s)P_{as}^*\right)^2\le \sum_{p|\delta}\sideset{}{'}{\sum}_{\substack{a\\p|a}}\frac{\mu^2(a)}{f_1^*(a)}\left(\sum_{s|\delta}\mu(s)P_{as}^*\right)^2.
\end{equation}
We consider the inner sum. By the Cauchy-Schwarz inequality we have
\begin{align}\sideset{}{'}{\sum}_{\substack{a\\p|a}}\frac{\mu^2(a)}{f_1^*(a)}\left(\sum_{s|\delta}\mu(s)P_{as}^*\right)^2&=\sideset{}{'}{\sum}_{\substack{a\\p|a}}\frac{\mu^2(a)}{f_1^*(a)}\left(\sum_{s|\delta/p}\mu(s)(P_{as}^*-P_{asp}^*)\right)^2\nonumber\\
&\ll d(\delta)\sum_{s|\delta/p}\sideset{}{'}{\sum}_{\substack{a\\p|a}}\frac{\mu^2(a)}{f_1^*(a)}\left(P_{as}^*-P^*_{asp}\right)^2.\end{align}
We split the summation over $a$ depending on whether the $P_{as}^*$ and $P_{asp}^*$ terms vanish (since $P_b^*=0$ for $b\ge R_2$).
\begin{align}
\sideset{}{'}{\sum}_{\substack{a\\p|a}}\frac{\mu^2(a)}{f_1^*(a)}\left(\sum_{s|\delta}\mu(s)P_{as}^*\right)^2&\ll d(\delta) \sum_{s|\delta/p}\sideset{}{'}{\sum}_{\substack{a'<R_2/s p^2}}\frac{\mu^2(a'p)}{f_1^*(a' p)}\left(P_{a' p s}^*-P^*_{a' s p^2}\right)^2\nonumber\\
&\qquad+d(\delta)\sum_{s|\delta/p}\sideset{}{'}{\sum}_{\substack{R_2/ s p^2\le a'<R_2/s p}}\frac{\mu^2(a' p)}{f_1^*(a' p)}(P_{a' ps}^*)^2.
\end{align}
We substitute in the value of $P^*$.
\begin{align}
\frac{1}{d(\delta)}\sideset{}{'}{\sum}_{\substack{a\\p|a}}&\frac{\mu^2(a)}{f_1^*(a)}\left(\sum_{s|q}\mu(s)P_{as}^*\right)^2\nonumber\\
&\ll (\log{R_2})^2\sum_{s|\delta/p}\sideset{}{'}{\sum}_{\substack{a'<R_2/s p^2}}\frac{\mu^2(a' p)}{f_1^*(a' p)}\left(\tilde{P}\left(1-\frac{\log{a' p s}}{\log{R_2}}\right)-\tilde{P}\left(1-\frac{\log{a' s p^2}}{\log{R_2}}\right)\right)^2\nonumber\\
&\qquad+(\log{R_2})^2\sum_{s|\delta/p}\sideset{}{'}{\sum}_{\substack{R_2/ s p^2\le a'<R_2/s p}}\frac{\mu^2(a' p)}{f_1^*(a' p)}\tilde{P}\left(1-\frac{\log{a'ps}}{\log{R_2}}\right)^2.
\end{align}
In the first sum above both the arguments of the polynomials differ by $\log{p}/\log{R_2}$. Since they are fixed polynomials, the derivative of the polynomial is $\ll 1$ and so the difference  is $\ll \log{p}/\log{R_2}$. In the second sum we just use the trivial bound $\tilde{P}(x)\ll 1$.

This gives
\begin{align}
\frac{1}{d(\delta)}\sideset{}{'}{\sum}_{\substack{a\\p|a}}\frac{\mu^2(a)}{f_1^*(a)}\left(\sum_{s|q}\mu(s)P_{as}^*\right)^2&\ll \frac{(\log{p})^2}{f_1^*(p)}\sum_{s|\delta/p}\sideset{}{'}{\sum}_{a<R_2/s p^2}\frac{\mu^2(a)}{f^*(a)}\nonumber\\
&\qquad+\frac{(\log{R_2})^2}{f(p)}\sum_{s|\delta/p}\sideset{}{'}{\sum}_{R_2/s p^2\le a<R_2/s p}\frac{\mu^2(a)}{f^*(a)}.
\end{align}
Using Lemma \ref{lmm:SummationLemma} we see that the first sum is $\ll d(\delta)(\log{p})^2(\log{R_2})^{k-1}/f_1^*(p)$ and the second sum is $\ll d(\delta)(\log{p})(\log{R_2})^{k}/f_1^*(p)$ because of the range of summation over $a$. Thus
\begin{equation}\sideset{}{'}{\sum}_{\substack{a\\p|a}}\frac{\mu^2(a)}{f_1^*(a)}\left(\sum_{s|\delta}\mu(s)P_{as}^*\right)^2\ll d(\delta)^2\frac{\log{p}}{f_1^*(p)}(\log{R_2})^{k}\ll d(\delta)^2\frac{\log{p}}{p}(\log{R_2})^{k}.\end{equation}
Summing over all $p| \delta$ gives the bound
\begin{equation}d(\delta)^2(\log{R_2})^k\sum_{p|\delta}\frac{\log{p}}{p}.\end{equation}
Splitting the sum into a sum over $p\le \log{R_2}$ and a sum over $p > \log{R_2}$ we get the bound
\begin{equation}d(\delta)^2(\log{R_2})^{k}(\log\log{R_2}).\end{equation}
This gives \eqref{eq:NotCoprimeSmall}, and hence the Lemma.
\end{proof}
Essentially the same argument as above also yields a useful bound on the size of $T_\delta$ and $T_\delta^*$.
\begin{lmm}\label{lmm:TSize}
Let $(\delta,A)=1$. Then we have
\begin{align*}
T_\delta&\ll \min_{p|\delta}\left(\log{p}\right)d(\delta)^2 (\log{R_2})^{k-1},\\
T_\delta^*&\ll \min_{p|\delta}\left(\log{p}\right)d(\delta)^2 (\log{R_2})^{k}+d(\delta)^2(\log{R_2})^{k}\log\log{R_2}.
\end{align*}
\end{lmm}
\begin{proof}
For $p|\delta$ we have (using the fact all terms are non-negative)
\begin{align}
T_\delta^*&=\sideset{}{'}{\sum}_{\substack{a\\(a,\delta)=1}}\frac{\mu^2(a)}{f_1^*(a)}\left(\sum_{s|\delta}\mu(s)P_{as}^*\right)^2+O(d(\delta)^2(\log{R_2})^{k}\log\log{R_2})\nonumber\\
&\ll d(\delta)\sideset{}{'}{\sum}_{a}\frac{\mu^2(a)}{f_1^*(a)}\sum_{s|\delta/p}\left(P_{as}^*-P_{asp}^*\right)^2+d(\delta)^2(\log{R_2})^{k}\log\log{R_2}\nonumber\\
&\ll d(\delta) (\log{R_2})^2\sum_{s|\delta/p}\sideset{}{'}{\sum}_{\substack{a<R_2/s p}}\frac{\mu^2(a)}{f_1^*(a)}\left(\tilde{P}\left(1-\frac{\log{a s}}{\log{R_2}}\right)-\tilde{P}\left(1-\frac{\log{a s p}}{\log{R_2}}\right)\right)^2\nonumber\\
&\qquad+d(\delta)(\log{R_2})^2\sum_{s|\delta/p}\sideset{}{'}{\sum}_{\substack{R_2/ s p\le a<R_2/s}}\frac{\mu^2(a)}{f_1^*(a)}\tilde{P}\left(1-\frac{\log{as}}{\log{R_2}}\right)^2\nonumber\\
&\qquad+d(\delta)^2(\log{R_2})^{k}\log\log{R_2}.
\end{align}
Noting the difference of the polynomials in the first sum is $\ll \log{p}/\log{R_2}$, and the polynomial in the second sum is $\ll 1$, we have
\begin{align}
T_\delta^*&\ll d(\delta)(\log{p})^2\sum_{s|\delta/p}\sideset{}{'}{\sum}_{\substack{a<R_2/s p}}\frac{\mu^2(a)}{f_1^*(a)}+d(\delta)(\log{R_2})^2\sum_{s|\delta/p}\sideset{}{'}{\sum}_{\substack{R_2/ s p\le a<R_2/s}}\frac{\mu^2(a)}{f_1^*(a)}\nonumber\\
&\qquad+d(\delta)^2(\log{R_2})^{k}\log\log{R_2}.
\end{align}
Appealing to Lemma \ref{lmm:SummationLemma} as in the previous lemma we obtain 
\begin{equation}T_\delta^*\ll d(\delta)^2(\log{p})(\log{R_2})^{k}+d(\delta)^2(\log{R_2})^{k}\log\log{R_2}.\end{equation}
The result for $T_\delta$ follows by a completely analogous argument. In this case the first line holds without the $O(d(\delta)^2(\log{R_2})^{k}\log\log{R_2})$ term, and so the final expression also holds without this term.
\end{proof}
With these results we are able to get an integral expression for $T_\delta$ and $T_\delta^*$ when $\delta$ has a bounded number of prime factors.
\begin{lmm}\label{lmm:TResult}
Let $p_1,\dots, p_{r-1}\nmid A$ for some primes $p_1,\dots, p_{r-1}$. Then we have
\begin{align*}
T_{p_1\dots p_{r-1}}&=(\log{R_2})^{k}\frac{\mathfrak{S}(\mathcal{L})}{(k-1)!}I_0\left(\frac{\log{p_1}}{\log{R_2}},\dots ,\frac{\log{p_{r-1}}}{\log{R_2}}\right)\\
&\qquad+O_r((\log{R_2})^{k-1}\log\log{R_2}),\nonumber\\
T_{p_1\dots p_{r-1}}^*&=(\log{R_2})^{k+1}\frac{\phi(A)\mathfrak{S}(\mathcal{L})}{A(k-2)!}I_1\left(\frac{\log{p_1}}{\log{R_2}},\dots ,\frac{\log{p_{r-1}}}{\log{R_2}}\right)\\
&\qquad+O_r((\log{R_2})^{k}\log\log{R_2}).\nonumber
\end{align*}
Here
\begin{align*}
I_0(x_1,\dots,x_{r-1})&=\int_0^1\left(\sum_{J\subset\{1,\dots,r-1\}}P^+\left(1-t-\sum_{j\in J}x_j\right)(-1)^{|J|}\right)^2 t^{k-1} d t,\\
I_1(x_1,\dots,x_{r-1})&=\int_0^1\left(\sum_{J\subset\{1,\dots,r-1\}}\tilde{P}^+\left(1-t-\sum_{j\in J}x_j\right)(-1)^{|J|}\right)^2 t^{k-2} d t,\\
P^+(x)&=\begin{cases}
P(x),\qquad &x\ge 0,\\
0,&\text{otherwise},
\end{cases}\\
\tilde{P}^+(x)&=\begin{cases}
\int_0^x P(t)d t,\qquad &x\ge 0,\\
0,&\text{otherwise.}\end{cases}
\end{align*}
\end{lmm}
\begin{proof}
Let $\delta=p_1\dots p_{r-1}$.

By Lemmas \ref{lmm:Tres} and \ref{lmm:Tdiv} we have that
\begin{equation}T_{\delta}^*=\sideset{}{'}{\sum}_a\frac{\mu^2(a)}{f_1^*(a)}\left(\sum_{s|\delta}\mu(s)P_{as}^*\right)^2+O_r\left((\log{R_2})^k\log\log{R_2}\right).\end{equation}
We recall from \eqref{eq:PDef} that for $a<R_2$ we have
\begin{equation}P_{a}^*=\mu^2(a)\frac{A}{\phi(A)}(\log{R_2})\mathfrak{S}(\mathcal{L})\tilde{P}^+\left(\frac{\log{R_2/a}}{\log{R_2}}\right).\end{equation}
Substituting this in above for $(\delta,A)=1$ we obtain 
\begin{align}
T_\delta^*&=\frac{A^2}{\phi(A)^2}(\log{R_2})^2\mathfrak{S}(\mathcal{L})^2\sum_{(a,A)=1}\frac{\mu^2(a)}{f_1^*(a)}\left(\sum_{s|\delta}\mu(s)\tilde{P}^+\left(\frac{\log{R_2/as}}{\log{R_2}}\right)\right)^2\nonumber\\
&\qquad+O_r\left((\log{R_2})^k\log\log{R_2}\right).
\end{align}
We again use Lemma \ref{lmm:SummationLemma} which shows that
\begin{equation}\sum_{\substack{a\le R_2\\(a,A)=1}}\frac{\mu^2(a)}{f_1^*(a)}\ll (\log{R_2})^{k-1}.\end{equation}
Thus
\begin{align}
T_{p_1\dots p_{r-1}}^*&=\frac{A^2}{\phi(A)^2}(\log{R_2})^2\mathfrak{S}(\mathcal{L})^2\sum_{(a,A)=1}\frac{\mu^2(a)}{f_1^*(a)}\left(\sum_{s|p_1\dots p_{r-1}}\mu(s)\tilde{P}^+\left(\frac{\log{R_2/as}}{\log{R_2}}\right)\right)^2\\
&\qquad+O_r\left((\log{R_2})^k(\log\log{R_2})\right).\nonumber
\end{align}
We also have
\begin{equation}T_{p_1\dots p_{r-1}}=\mathfrak{S}(\mathcal{L})^2\sum_{(a,A)=1}\frac{\mu^2(a)}{f_1(a)}\left(\sum_{s|p_1\dots p_{r-1}}\mu(s)P^+\left(\frac{\log{R_2/as}}{\log{R_2}}\right)\right)^2+O_r\left((\log{R_2})^{k-1}\right).\end{equation}
We can now estimate the main term using \cite{GGPY}[Lemma 4]. First we put
\begin{align}
\gamma(p)&=\begin{cases}
k-1,\qquad &p\nmid A\nonumber\\
0,&\text{otherwise}.\end{cases}\\
g(d)&=\prod_{p|d}\frac{\gamma(p)}{p-\gamma(p)},\\
F(t)&=F_{x_1,\dots,x_{r-1}}(t)=\sum_{J\subset\{1,\dots,r-1\}}(-1)^{|J|}\tilde{P}^+\left(t+\sum_{j\in J}x_j\right).\nonumber
\end{align}
If we put $x_i=\log{p_i}/\log{R_2}$ for each $i\in \{1,\dots,r-1\}$ then we see that
\begin{equation}\sum_{(a,A)=1}\frac{\mu^2(a)}{f_1^*(a)}\left(\sum_{s|p_1\dots p_{r-1}}\mu(s)\tilde{P}^+\left(\frac{\log{R_2/as}}{\log{R_2}}\right)\right)^2=\sum_{d\le R_2}\mu^2(d)g(d)F\left(\frac{\log{R_2/d}}{\log{R_2}}\right).\end{equation}
Since $F$ is a continuous piecewise differentiable function we can apply \cite{GGPY}[Lemma 4] which gives
\begin{equation}\sum_{d\le R_2}\mu^2(d)g(d)F\left(\frac{\log{R_2/d}}{\log{R_2}}\right)=\frac{A}{\phi(A)}\frac{(\log{R_2})^{k-1}}{\mathfrak{S}(\mathcal{L})(k-2)!}\int_0^1 F(1-t)t^{k-2}d t+O\left((\log{R_2})^{k-2}\right).\end{equation}
Similarly we follow the same procedure instead with
\begin{align}
\gamma(p)&=\begin{cases}
k,\qquad&p\nmid A\\
0,&\text{otherwise},\end{cases}\nonumber\\
G(t)&=\sum_{J\subset\{1,\dots,r-1\}}(-1)^{|J|}P^+\left(t+\sum_{j\in J}x_j\right).
\end{align}
This yields
\begin{align}
\sum_{(a,A)=1}\frac{\mu^2(a)}{f_1(a)}&\left(\sum_{s|p_1\dots p_{r-1}}\mu(s)P^+\left(\frac{\log{R_2/as}}{\log{R_2}}\right)\right)^2\nonumber\\
&=\frac{(\log{R_2})^k}{\mathfrak{S}(\mathcal{L})(k-1)!}\int_0^1G(1-t)^2 t^{k-1}dt+O\left((\log{R_2})^{k-1}\right).
\end{align}
\end{proof}
We also require a bound on the size of the sieve coefficients $\lambda_d$.
\begin{lmm}\label{lmm:LambdaSize}
We have that
\[\lambda_d\ll (\log{R_2})^{k}.\]
\end{lmm}
\begin{proof}
This is proven in \cite{GGPY}[Proof of Theorem 7].
\end{proof}
We finish this section with a partial summation lemma, which will be useful later on.
\begin{lmm}\label{lmm:PrimeSummation}
Let $0\le a<b$ be fixed constants. Let  $V:[a,b]\rightarrow \mathbb{R}_{\ge 0}$ be a continuous piecewise smooth function. If $V$ satisfies $V(x)\ll x$ uniformly for $x\in [a,b]$ then we have
\begin{equation*}\sum_{R^a\le p\le R^b}\frac{1}{p}V\left(\frac{\log{p}}{\log{R}}\right)=\int_{a}^b \frac{V(u)}{u}d u+O\left(\frac{M(V)\log\log{R}}{\log{R}}\right),\end{equation*}
where
\begin{equation*}
M(V)=\sup_{t\in[a,b]}\left(1+|V'(t)|\right).
\end{equation*}
\end{lmm}
\begin{proof}
The result follows straightforwardly by partial summation and the prime number theorem.

If $a=0$ then we replace $a$ with $2/\log{R}$. This leaves the left hand side of the result unchanged, and introduces an error
\begin{equation}
\int_{a}^{2/\log{R}}\frac{V(u)}{u}d u \ll \frac{1}{\log{R}}
\end{equation}
to the right hand side, which can be absorbed into the error term.

By the prime number theorem
\begin{equation}\pi(y)=y\left(1+O\left(\frac{1}{\log{y}}\right)\right).\end{equation}
Therefore, by partial summation we have
\begin{align}
\sum_{R^a\le p_j\le R^b}\frac{1}{p}V\left(\frac{\log{p}}{\log{R}}\right)&=O\left(\frac{1}{\log{R}}\right)+\int_{R^a}^{R^b}\frac{t}{t^2\log{t}}V\left(\frac{\log{t}}{\log{R}}\right)\left(1+O\left(\frac{1}{\log{t}}\right)\right)d t\nonumber\\
&\qquad+\int_{R^a}^{R^b}\frac{t}{t^2(\log{t})(\log{R})}V'\left(\frac{\log{t}}{\log{R}}\right)\left(1+O\left(\frac{1}{\log{t}}\right)\right)d t\nonumber\\
&=\int_a^b\frac{V(u)}{u}d u + O\left(\int_a^b\frac{1+|V'(u)|}{u\log{R}}d u \right) +O\left(\frac{1}{\log{R}}\right)\nonumber\\
&=\int_a^b\frac{V(u)}{u}d u + O\left(\frac{M(V)\log\log{R}}{\log{R}}\right).
\end{align}
\end{proof}
\newpage
\section{Proof of Proposition \ref{prpstn:MainTermResult}}
We consider the weighted sum of Proposition \ref{prpstn:MainTermResult} in a similar way to previous work on Selberg's $\Lambda^2\Lambda^-$ sieve which in its basic form considers the weight $W_0(x)=-1$.
\begin{align}
\sum_{N\le n\le 2N}\left(\sum_{\substack{p|\Pi(n)\\p\le R_1}}W_0\left(\frac{\log{p}}{\log{R_2}}\right)\right)\left(\sum_{\substack{d|\Pi(n)\\d\le R_2}}\lambda_d\right)^2&=\sum_{p\le R_1}W_0\left(\frac{\log{p}}{\log{R_2}}\right)\sum_{d,e\le R_2}\lambda_d\lambda_e\sum_{\substack{N\le n\le 2N\\ [p,d,e]|\Pi(n)}}1\nonumber\\
&=N\sideset{}{'}{\sum}_{p\le R_1}W_0\left(\frac{\log{p}}{\log{R_2}}\right)\sideset{}{'}{\sum}_{d,e\le R_2}\frac{\lambda_d\lambda_e}{f([d,e,p])}+O_{W_0}\left(E_1\right)\nonumber\\
&=N\sideset{}{'}{\sum}_{p\le R_1}W_0\left(\frac{\log{p}}{\log{R_2}}\right)\frac{T_p}{f(p)}+O_{W_0}(E_1),
\end{align}
where
\begin{equation}
E_1=\sum_{p\le R_1}\sum_{d,e\le R_2}\left|\lambda_d\lambda_er_{[d,e,p]}\right|,\qquad  r_d=\sum_{\substack{N<n\le 2N\\ d|\Pi(n)}}1-\frac{N}{f(d)}.
\end{equation}
By Lemma \ref{lmm:LambdaSize} we have $\lambda_d\ll (\log{N})^k$, and we note that $r_d\le k^{\omega(d)}$. Therefore we have
\begin{align}
E_1&\ll (\log{N})^{2k}\sum_{\substack{p\le R_1\\ d,e\le R_2}}\mu^2([d,e,p])k^{\omega([d,e,p])}\nonumber\\
&\ll  (\log{N})^{2k}\sum_{r\le R_2^2R_1}\mu^2(r)(7k)^{\omega(r)}\nonumber\\
&\ll  (\log{N})^{2k}R_2^2R_1\sum_{r\le R_2^2R_1}\frac{\mu^2(r)(7k)^{\omega(r)}}{r}\nonumber\\
&\ll (\log{N})^{2k}R_2^2R_1\prod_{p\le R_2^2R_1}\left(1+\frac{7k}{p}\right)\nonumber\\
&\ll (\log{N})^{9k}R_2^2R_1.
\end{align}
Thus for $R_2^2R_1\le N(\log{N})^{-9k}$ we have $E_1\ll N$.

By Lemma \ref{lmm:TResult} we have
\begin{equation}T_p=(\log{R_2})^k\frac{\mathfrak{S}(\mathcal{L})}{(k-1)!}I_0\left(\frac{\log{p}}{\log{R_2}}\right)+O\left((\log{N})^{k-1}\log\log{N}\right),\end{equation}
where
\begin{equation}I_0(x)=\int_0^1\left(P^+_1(1-t)-P^+_1(1-t-x)\right)^2t^{k-1}d t.\end{equation}
Recalling that $f(p)=p/k$ for $p\nmid A$, we see that the error terms from $T_p$ contribute
\begin{equation}
\ll_{W_0}(\log{N})^{k-1}\log\log{N}\sum_{p\le R_1}\frac{1}{p}\ll (\log{N})^{k-1}(\log\log{N})^2.
\end{equation}
Therefore we are left to estimate the sum
\begin{equation}\sideset{}{'}{\sum}_{p\le R_1}\frac{1}{p}W_0\left(\frac{\log{p}}{\log{R_2}}\right)I_0\left(\frac{\log{p}}{\log{R_2}}\right).\end{equation}
We note that if $t\le 1-x$ then $P^+(1-t)-P^+(1-t-x)\ll x$, and so
\begin{equation}\label{eq:I0Bound}I_0(x)\ll x.\end{equation}
If $1-x\le t\le 1$ then since the interval has length $x$ we also have
\begin{equation}I_0(x)\ll x.\end{equation}
By the piecewise smoothness of $I_0(x)$ and $W_0(x)$ we have uniformly for $x\in[0,r1/r2]$
\begin{equation}I_0'(x)\ll 1,\qquad W_0'(x)\ll_{W_0}1.\end{equation}
Therefore by Lemma \ref{lmm:PrimeSummation},  we have
\begin{equation}\sum_{p\le R_1}\frac{1}{p}W_0\left(\frac{\log{p}}{\log{R_2}}\right)I_0\left(\frac{\log{p}}{\log{R_2}}\right)=\int_0^{r_1/r_2}\frac{W_0(u)}{u}I_0(u)du+O_{W_0}\left(\frac{\log\log{N}}{\log{N}}\right).
\end{equation}
By \eqref{eq:I0Bound} we see that the contribution to the above sum for primes which divide $A$ is
\begin{equation}\ll \frac{1}{\log{N}}.\end{equation}
This gives the result.
\section{Proof of Proposition \ref{prpstn:ErResult}}
We will follow a similar argument to that of Graham, Goldston, Pintz and Y\i ld\i r\i m\cite{GGPY} where the result was obtained with $r=2$ and $W_2(x_1,x_2)=1$. Thorne \cite{Thorne} extended this in the natural way to consider $r>2$, again without the weighting $W_r$. In order to introduce the weighting by $W_r$, it is necessary to establish a Bombieri-Vinogradov style result for numbers with $r$ prime factors weighted by $W_r$.
\begin{lmm}\label{lmm:BombVino}
Let
\begin{equation*}\beta_r(n)=\begin{cases}
W_r\left(\frac{\log{p_1}}{\log{n}},\dots,\frac{\log{p_{r-1}}}{\log{n}}\right),\qquad &\text{$n=p_1p_2\dots p_r$ with $p_1\le \dots\le p_r$,}\\
0,&\text{otherwise,}
\end{cases}\end{equation*}
for some piecewise smooth function $W_r:[0,1]^{r-1}\rightarrow\mathbb{R}$.

Put
\begin{equation}\Delta_{\beta,r}(x;q)=\max_{y\le x}\max_{\substack{a\\ (a,q)=1}}\left|\sum_{\substack{y<n\le 2y\\ n \equiv a \pmod{q}}}\beta_r(n)-\frac{1}{\phi(q)}\sum_{\substack{y<n\le 2y\\(n,q)=1}}\beta_r(n)\right|\end{equation}
For every fixed integer $h>0$, and for every $C>0$ there exists a constant $C'=C'(C,h)$ such that if $Q\le x^{1/2}(\log{x})^{-C'}$ then we have
\begin{equation}\sum_{q\le Q}\mu^2(q)h^{\omega(q)}\Delta_{\beta,r}(x;q)\ll _{C,h,W_r} x(\log{x})^{-C}.\end{equation}
\end{lmm}
\begin{proof}
This result follows from the Bombieri-Vinogradov theorem for numbers with exactly $r$ prime factors, as proven by Motohashi \cite{Motohashi}, and the continuity of $W_r$.

We assume that $W_r$ is smooth. The result can be extended to piecewise smooth functions by taking smooth approximations.

We fix a constant $C>0$, an integer $h$, and a function $W_r$.

We let
\begin{equation}\chi_{\textbf{$\delta$},\textbf{$\eta$}}(n)=\begin{cases}
1,\qquad &\text{$n=p_1p_2\dots p_r$ with $n^{\eta_i}\le p_i\le n^{\delta_i}$ $\forall i$}\\ 
&\qquad\text{and }p_1<p_2<\dots<p_{r}.\\
0,&\text{otherwise}.
\end{cases}
\end{equation}
By Motohashi's result \cite{Motohashi}[Theorem 2] we have that uniformly for any choice of constants $\delta_i$ and $\eta_i$ ($i=1,\dots ,r$) there is a constant $C'=C'(C,h)$ such that if $Q\le x^{1/2}(\log{x})^{-C'}$ then we have
\begin{equation}\sum_{q\le Q}\mu^2(q)h^{\omega(q)}\max_{\substack{y,a\\y\le x\\ (a,q)=1}}\left|\sum_{\substack{y\le n\le 2y\\n\equiv a \pmod{q}}}\chi_{\delta,\eta}(n)-\frac{1}{\phi(q)}\sum_{\substack{y \le n \le 2y\\ (n,q)=1}}\chi_{\delta,\eta}(n)\right|\ll_{C,h} x(\log{x})^{-(C+h)(r+1)}.\end{equation}
We choose $\delta_i\in\{(\log{x})^{-C-h},2(\log{x})^{-C-h},\dots, \lceil(\log{x})^{C+h}\rceil (\log{x})^{-C-h}\}$ separately for each $i\in \{1,\dots,r\}$, subject to the constraint $\delta_i\le\delta_{i+1}$ ($1\le i\le r-1$). For each choice of the $\delta_i$ we take $\eta_i=\delta_i-(\log{x})^{-C-h}$ for $1\le i\le r$. We put
\begin{equation}W_r(\delta)=W_r(\delta_1,\delta_2,\dots,\delta_{r-1}).\end{equation}
We notice that by the smoothness of $W_r$ we have that
\begin{align}
\beta_r(n)&=\sum_{\delta}\chi_{\delta,\eta}(n)\left(W_r(\delta)+O((\log{x})^{-C-h})\right)\nonumber\\
&=\sum_\delta \chi_{\delta,\eta}(n)W_r(\delta)+O\left((\log{x})^{-C-h}\right).
\end{align}
Here $\sum_{\delta}$ indicates a sum over all the $O((\log{x})^{r(C+h)})$ possible choices of the $\delta_i$. 

Therefore we have that
\begin{align}
\sum_{\substack{y\le n\le 2y\\ n\equiv a \pmod{q}}}\beta_r(n)-\frac{1}{\phi(q)}\sum_{\substack{y\le n\le 2y\\ (n,q)=1}}\beta_r(n)&=\sum_{\delta}W_r(\delta)\left(\sum_{\substack{y\le n\le 2y\\ n\equiv a \pmod{q}}}\chi_{\delta,\eta}(n)-\frac{1}{\phi(q)}\sum_{\substack{y\le n\le 2y\\ (n,q)=1}}\chi_{\delta,\eta}(n)\right)\nonumber\\
&\qquad+O\left((\log{y})^{-C-h}\frac{y}{\phi(q)}\right).
\end{align}
Thus for $Q\le x(\log{x})^{-C'}$ we have
\begin{align}
\sum_{q\le Q}\mu^2(r)&h^{\omega(q)}\Delta_{\beta,r}(x;q)\nonumber\\
&\le\sum_{\delta}W_r(\delta)\sum_{q\le Q}\mu^2(r)h^{\omega(q)}\max_{\substack{a,y\\y\le x\\(a,q)=1}}\left|\sum_{\substack{y\le n\le 2y\\n\equiv a \pmod{q}}}\chi_{\delta,\eta}(n)-\frac{1}{\phi(q)}\sum_{\substack{y \le n \le 2y\\ (n,q)=1}}\chi_{\delta,\eta}(n)\right|\nonumber\\
&\qquad+O\left((\log{x})^{-(C+h)}\sum_{q\le Q}\mu^2(q)h^{\omega(q)}\frac{x}{\phi(q)}\right)\nonumber\\
&\ll \sum_{\delta}W_r(\delta)x(\log{x})^{-(C+h)(r+1)}+x(\log{x})^{-(C+h)}\prod_{p\le Q}\left(1+\frac{h}{p-1}\right)\nonumber\\
&\ll x(\log{x})^{-C}.
\end{align}
\end{proof}
With this, we can adapt the argument of Thorne \cite{Thorne} slightly to rewrite the main term in terms of the quantities $T_q^*$.
\begin{lmm}
We have
\begin{align*}\sum_{N\le n\le 2N}\beta_r(L_j(n))\left(\sum_{d|\Pi(n)}\lambda_d\right)^2&=\frac{AN}{\phi(A)(\log{N})}\sum_{\substack{p_1,\dots,p_{r-1}\\ N^\epsilon<p_1<p_2<\dots<p_{r-1}\\ q<\min(N/R_2,N/p_{r-1})}}\frac{T^*_{q}}{q}\alpha\left(\frac{\log{p_1}}{\log{R_2}},\dots,\frac{\log{p_{r-1}}}{\log{R_2}}\right)\\
&\qquad+O_{W_r}\left(N(\log{N})^{k-1}(\log\log{N})^{r-1}\right),\end{align*}
where
\begin{align*}
q&=\prod_{i=1}^{r-1}p_i,\\
T_\delta^*&=\sum_{\substack{d,e\\(d,A)=(e,A)=1}}\frac{\lambda_d \lambda_e}{f^*([d,e,\delta]/\delta)},\\
\alpha(q)&=\left(\frac{\log{N}}{\log{N}-\log{q}}\right)W_r\left(\frac{\log{p_1}}{\log{R_2}},\dots,\frac{\log{p_{r-1}}}{\log{R_2}}\right).
\end{align*}
\end{lmm}
\begin{proof}
Thorne \cite{Thorne} considers essentially the same sum but without the weighting by $W_r$. In his argument up until equation (4.14) on Page 15, this difference only affects the argument when he appeals to the Bombieri-Vinogradov theorem for $E_h$ numbers (where $h\le r$). Lemma \ref{lmm:BombVino} gives the equivalent Bombieri-Vinogradov style result when weighting by $W_r$, and so exactly the same argument follows through. The only additional assumption of Thorne is that he restricts the consideration to numbers $n=p_1\dots p_r$ satisfying
\begin{equation}\exp(\sqrt{\log{N}})<p_1<\dots<p_{r}\qquad\text{and}\qquad R_2<p_r.\end{equation}
This is satisfied if for a fixed $\epsilon>0$ we require $W_r$ to be supported on
\begin{equation}
\mathcal{A}_r=\left\{x\in[0,1]^{r-1}:\epsilon<x_1<\dots<x_{r-1},\sum_{i=1}^{r-1}x_i<\min(1-r_2,1-x_{r-1})\right\}.
\end{equation}
This gives us in our case (the equivalent of Thorne's equation (4.14) but with the explicit error term he calculates) 
\begin{align}
\sum_{N\le n\le 2N}\beta_r(L_j(n))\left(\sum_{d|\Pi(n)}\lambda_d\right)^2&=\sideset{}{'}{\sum}_{d,e}\lambda_d\lambda_e\sideset{}{^*}{\sum}_{p_1,\dots,p_{r-1}}\frac{d_{k-1}([d,e,q]/q)}{\phi(a_j[d,e,q]/q)}\nonumber\\
&\qquad\times \sum_{a_jN/q\le m\le 2a_jN/q}1_{\mathbb{P}}(n)W_r\left(\frac{\log{p_1}}{\log{mq}},\dots,\frac{\log{p_{r-1}}}{\log{mq}}\right)\nonumber\\
&\qquad+O(N).
\end{align}
Here and from now we use the symbol $\displaystyle{\sideset{}{^*}{\sum}}$ to indicate that we are summing over primes $p_1,\dots,p_{r-1}$ with
\begin{equation}
\left(\frac{\log{p_1}}{\log{N}},\dots,\frac{\log{p_{r-1}}}{\log{N}}\right)\in\mathcal{A}_r.
\end{equation}
Again we assume for simplicity that $W_r$ is smooth. By taking smooth approximations one can establish the result for piecewise-smooth $W_r$.

Estimating the inner sum gives
\begin{align}
\sum_{a_j N/q\le m\le 2a_j N/q}&1_{\mathbb{P}}(m)W_r\left(\frac{\log{p_1}}{\log{mq}},\dots,\frac{\log{p_{r-1}}}{\log{mq}}\right)\nonumber\\
&=\left(W_r\left(\frac{\log{p_1}}{\log{N}},\dots,\frac{\log{p_{r-1}}}{\log{N}}\right)+O\left(\frac{1}{\log{N}}\right)\right)\left(\pi\left(\frac{2a_j N}{q}\right)-\pi\left(\frac{a_j N}{q}\right)\right)\nonumber\\
&=W_r\left(\frac{\log{p_1}}{\log{N}},\dots,\frac{\log{p_{r-1}}}{\log{N}}\right)\frac{a_j N}{\log{N}}\left(\frac{\log{N}}{\log{N}-\log{q}}\right)\left(1+O\left(\frac{1}{\log{N}}\right)\right).
\end{align}
We note that by Hypothesis \ref{hypthss:Normalised} if $d|\Pi(n)$ then $(d,A)=1$. Therefore $(a_j,[d,e,q]/q)=1$, so $\phi(a_j[d,e,q]/q)=\phi(a_j)\phi([d,e,q]/q)$. Together these give
\begin{align}
\sum_{N\le n\le 2N}&\beta_r(L_j(n))\left(\sum_{d|\Pi(n)}\lambda_d\right)^2\nonumber\\
&=\frac{a_j N}{\phi(a_j)\log{N}}\sideset{}{^*}{\sum}_{p_1,\dots,p_{r-1}}\frac{T^*_{q}W_r\left(\frac{\log{p_1}}{\log{N}},\dots,\frac{\log{p_{r-1}}}{\log{N}}\right)\log{N}}{q(\log{N}-\log{q})}(1+O((\log{N})^{-1}))\nonumber\\
&\qquad\qquad+O(N)\nonumber\\
&=\frac{a_j N}{\phi(a_j)\log{N}}\sideset{}{^*}{\sum}_{p_1,\dots,p_{r-1}}\frac{T^*_{q}}{q}\alpha\left(\frac{\log{p_1}}{\log{R_2}},\dots,\frac{\log{p_{r-1}}}{\log{R_2}}\right)\left(1+O\left(\frac{1}{\log{N}}\right)\right)\nonumber\\
&\qquad\qquad+O(N),
\end{align}
where
\begin{equation}
\alpha(x_1,\dots,x_{r-1})=\frac{W_r(r_2x_1,\dots,r_2x_{r-1})}{1-r_2\sum_{i-1}^{r-1}x_i}.
\end{equation}
We note that $a_j$ and $A$ are composed of the same prime factors, so $a_j/\phi(a_j)=A/\phi(A)$. Therefore the main term is that of the Lemma.

By Lemma \ref{lmm:TSize} we have
\begin{equation}
T_q^*\ll_r (\log{N})^k\log{p_1}+(\log{N})^k\log\log{N}.
\end{equation}
We also have
\begin{equation}\alpha(x_1,\dots,x_{r-1})\ll_{W_r}1.\end{equation}
Thus the $O(1/\log{N})$ term contributes
\begin{align}
\ll_{W_r,r}N(\log{N})^{k-2}\sideset{}{^*}{\sum}_{p_1,\dots,p_{r-1}}\frac{\log{p_1}+\log\log{N}}{p_1\dots p_{r-1}}\ll_{W_r,r}N(\log{N})^{k-1}(\log\log{N})^{r-2}.
\end{align}
This gives the result.
\end{proof}
\begin{lmm}
We have
\begin{align}
\sideset{}{^*}{\sum}_{p_1,\dots,p_{r-1}}&\frac{T^*_{q}}{q}\alpha\left(\frac{\log{p_1}}{\log{R_2}},\dots,\frac{\log{p_{r-1}}}{\log{R_2}}\right)\nonumber\\
&=(\log{R_2})^{k+1}\frac{\phi(A)\mathfrak{S}(\mathcal{L})}{A(k-2)!}\int\dots\int \frac{I_1(u_1,\dots,u_{r-1})\alpha(u_1,\dots,u_{r-1})}{u_1u_2\dots u_{r-1}}du_1\dots du_{r-1}\nonumber\\
&\qquad+O\left((\log\log{N})^{r}(\log{N})^{k}\right)
\end{align}
Where the integration is subject to the constraints
\begin{align}
\epsilon<u_1<\dots<u_{r-1},\qquad \text{and}\qquad \sum_{i=1}^{r-1}u_i\le \min(r_2^{-1}-1,r_2^{-1}-u_{r-1}).
\end{align}
\end{lmm}
\begin{proof}
By Lemma \ref{lmm:TResult} for $q=p_1p_2\dots p_{r-1}$ we have
\begin{equation}T_q^*=(\log{R_2})^{k+1}\frac{\phi(A)\mathfrak{S}(\mathcal{L})}{A(k-2)!}I_1\left(\frac{\log{p_1}}{\log{R_2}},\dots ,\frac{\log{p_{r-1}}}{\log{R_2}}\right)+O_r((\log{N})^{k}\log\log{N}).\end{equation}
Thus summing the error term over $p_1,\dots p_{r-1}$ gives a contribution
\begin{align}
&\sideset{}{^*}{\sum}_{p_1,\dots,p_{r-1}}\frac{1}{q}\alpha\left(\frac{\log{p_1}}{\log{R_2}},\dots,\frac{\log{p_{r-1}}}{\log{R_2}}\right)(\log{N})^k\log\log{N}\nonumber \\
&\ll_{W_r} (\log{N})^{k}\log\log{N}\left(\sum_{p\le N}\frac{1}{p}\right)^{r-1}\nonumber \\
&\ll_{W_r} (\log{N})^{k}(\log\log{N})^r.
\end{align}
We are therefore left to evaluate the main term
\begin{equation}
\sideset{}{^*}{\sum}_{p_1,\dots,p_{r-1}}\frac{1}{q}\alpha\left(\frac{\log{p_1}}{\log{R_2}},\dots,\frac{\log{p_{r-1}}}{\log{R_2}}\right)I_1\left(\frac{\log{p_1}}{\log{R_2}},\dots ,\frac{\log{p_{r-1}}}{\log{R_2}}\right).
\end{equation}
We will now apply Lemma \ref{lmm:PrimeSummation} to $p_{r-1},\dots,p_1$ in turn to estimate the sum $\sideset{}{^*}{\sum}\alpha(q)T_q^*q^{-1}$.

For $u_1,\dots,u_j\in [0,r_2^{-1}]$ we put
\begin{equation}V_j(u_1,\dots,u_j)=\idotsint \frac{1}{\prod_{i=j+1}^{r-1}u_i}\alpha\left(u_1,\dots,u_{r-1}\right)I_1\left(u_{1},\dots,u_{r-1}\right)d u_{j+1}\dots d u_{r-1},\end{equation}
where the integration is subject to $u_j<u_{j+1}<\dots<u_{r-1}$ and $\sum_{1}^{r-1}u_i\le \min(r_2^{-1}-1,r_2^{-1}-u_{r-1})$.

As in the proof of Lemma \ref{lmm:TSize}, since $\tilde{P}$ is continuous and its derivative is uniformly bounded on $[0,1]$, we have that
\begin{align}
I_1(u_1,&\dots,u_{r-1})=\int_0^1\left(\sum_{J\subset\{1,\dots,r-1\}}\tilde{P}^+\left(1-t-\sum_{i\in J}u_i\right)(-1)^{|J|}\right)^2 t^{k-2} d t\nonumber\\
&\ll_r \int_0^1\sum_{J\subset\{1,\dots,r-1\}\backslash \{j\}}\left(\tilde{P}^+\left(1-t-\sum_{i\in J}u_i\right)-\tilde{P}^+\left(1-t-u_j-\sum_{i\in J}u_i\right)\right)^2 t^{k-2} d t\nonumber\\
&\ll_r u_j^2.
\end{align}
Thus, since $\alpha(u_1,\dots,u_{r-1})\ll 1$, we have uniformly for $u_1,\dots,u_j\in[0,r_2^{-1}]$
\begin{align}
V_j(u_1,\dots,u_j)&\ll u_j^2\int\dots\int\frac{1}{\prod_{i=j+1}^{r-1}u_i}du_{j+1}\dots du_{r-1}\nonumber\\
&\ll u_j^2(1+|\log{1/u_j}|^r)\nonumber\\
&\ll u_j.
\end{align}
Moreover, essentially the same argument shows that uniformly for $u_1,\dots,u_j\in[0,r_2^{-1}]$ we have
\begin{equation}
\frac{\partial}{\partial u_j}I_1(u_1,\dots,u_{r-1})\ll_r u_j.
\end{equation}
Thus since
\begin{equation}
\frac{\partial}{\partial u_j}\alpha(u_1,\dots,u_{r-1})\ll_r 1
\end{equation}
we have that
\begin{align}
\frac{\partial}{\partial u_j}V_j(u_1,\dots,u_j)&\ll u_j\idotsint\frac{1}{\prod_{i=j+1}^{r-1}u_i}du_{j+1}\dots du_{r-1}\nonumber\\
&\ll 1.
\end{align}
Thus the condition of Lemma \ref{lmm:PrimeSummation} applies for the function $V_j$. Applying Lemma \ref{lmm:PrimeSummation} in turn to $V_{r-1},V_{r-2},\dots,V_{1}$ gives the result. We note that the error terms contribute a total which is $\ll (\log{N})^{k}(\log\log{N})^{r-1}$.
\end{proof}

\section{Proof of Proposition \ref{prpstn:SquareFree}}
By Lemma \ref{lmm:LambdaSize} we have $\lambda_d\ll (\log{N})^k$. Therefore we have 
\begin{align}
\sum_{p\le AN^{1/2}}\sum_{\substack{N\le n \le 2N\\ p^2|\Pi(n)}}\left(\sum_{\substack{d|\Pi(n)\\ d\le R_2}}\lambda_d\right)^2&=N\sum_{p\le AN^{1/2}}\sideset{}{'}{\sum}_{d,e\le R_2}\frac{\lambda_d\lambda_e}{f([d,e,p^2])}+O\left(\sum_{p\le AN^{1/2}}\sum_{d,e\le R_2}|\lambda_d\lambda_er_{[d,e,p^2]}|\right)\nonumber\\
&\ll N\frac{T_{p}}{p^2}+O\left((\log{N})^{2k}\sum_{r\le R_2^2AN^{1/2}}\mu^2(r)(7k)^{\omega(r)}\right).
\end{align}
We first bound the error term
\begin{align}
\sum_{r\le R_2^2AN^{1/2}}\mu^2(r)(7k)^{\omega(r)}&\ll R_2^2N^{1/2}\sum_{r\le AR_2^2N^{1/2}}\frac{\mu^2(r)(7k)^{\omega(r)}}{r}\nonumber\\
&\ll R_2^2N^{1/2}\prod_{p\le AR_2^2N^{1/2}}\left(1+\frac{7k}{p}\right)\nonumber\\
&\ll R_2^2N^{1/2}(\log{N})^{7k}.
\end{align}
Thus for $R_2\le N^{1/4}(\log{N})^{-5k}$ the error term is $O(N)$.

By Lemma \ref{lmm:TSize} we have that
\begin{equation}T_{p}\ll (\log{N})^{k-1}\log{p}+(\log{N})^{k-1}\log\log{N}.\end{equation}
Thus
\begin{align}
\sum_{p\le AN^{1/2}}\sum_{\substack{N\le n \le 2N\\ p^2|\Pi(n)}}\left(\sum_{\substack{d|\Pi(n)\\ d\le R_2}}\lambda_d\right)^2&\ll N(\log{N})^{k-1}\sum_{p\le N^{1/4}}\frac{\log{p}+\log\log{N}}{p^2}+O(N)\nonumber\\
&\ll N(\log{N})^{k-1}\log\log{N}.
\end{align}
\section{Acknowledgment}
I would like to thank my supervisor, Prof. Heath-Brown and Dr. Craig Franze for many helpful comments.
\bibliographystyle{acm}
\bibliography{bibliography}
\end{document}